\documentclass{article}

\usepackage[french,english]{babel}
\usepackage[T1]{fontenc}
\usepackage[utf8]{inputenc}
\usepackage{lmodern}
\usepackage{amsmath,amsthm,amssymb, amsfonts}
\usepackage[left=1in, right=1in, top=1in, bottom=1in]{geometry} 
\usepackage{indentfirst}
\usepackage{fourier}
\usepackage{bm}
\usepackage{float}
\usepackage[colorlinks]{hyperref}
\usepackage{csquotes}
\usepackage{array}
\newtheorem{theorem}{Theorem}[section]

\newtheorem{theoreme}{Théorème}[section]

\newtheorem{remark}{Remark}[section]
\theoremstyle{remark}
\newtheorem{ex}{Example}[section]
\theoremstyle{definition}

\newcommand{\W}{\mathcal{W}}
\newcommand{\R}{\mathbb{R}}

\newcommand{\dr}{\mathrm{d}}

\title{Quantitative sensitivity analysis for Fokker-Planck equation with respect to the Wasserstein distance}
\author{Martin Morange\footnote{Inria, CMAP, CNRS, École polytechnique, Institut Polytechnique de Paris, 91120 Palaiseau, France. Mail: martin.morange@inria.fr}}
\date{}

\begin{document}

\maketitle

\selectlanguage{english}
\begin{abstract}
	We analyze the sensitivity of solutions to the Fokker-Planck equation with respect to some unknown parameter. Our main result is to provide quantitative upper bounds for the $p$-Wasserstein distance $\W_p$ between two solutions with different parameters, for every $p \geq 2$. We are able to give two proofs of this result, the first relying on synchronous coupling between two solutions of an SDE, and another one that relies on the differentiation of Kantorovitch dual formulation of optimal transport. We also provide more specific bounds in the case of the overdamped Langevin process, for which we are able to compare convergence to the invariant measure and sensitivity to the parameter.
\end{abstract}

\selectlanguage{french}

\begin{abstract}
    Nous analysons la sensibilité des solutions de l'équation de Fokker-Planck par rapport à un paramètre inconnu. Notre principal résultat est de fournir des bornes supérieures quantitatives pour la distance de Wasserstein d'ordre $p$ $\W_p$ entre deux solutions correspondant à des paramètres différents. Nous donnons deux preuves de ce résultat, une première reposant sur un couplage synchrone entre deux solutions à une EDS, et une autre qui repose sur la différentiation de la formulation duale de Kantorovitch du transport optimal. Nous considérons aussi le processus de Langevin dans un second temps, pour lequel nous prouvons des bornes légerement différentes afin de comparer convergence vers une mesure invariante et sensibilité au paramètre. 
\end{abstract}

\section*{Version française abrégée}

Nous considérons une équation de Fokker-Planck dépendant d'un paramètre $a \in \R^p$, avec un drift $b : \R^d \times \R^p \rightarrow \R^d$ et un coefficient de diffusion $A : \R^d \times \R^p \rightarrow \mathbb{S}_{++}^{d}(\R)$,

\begin{equation}
\begin{cases}
    \partial_t \rho(t,x,a) =  \nabla_x \cdot \left(A(x,a) \nabla_x \rho(t,x,a) \right) - \nabla_x  \cdot \left(\rho(t,x,a) b(x,a)  \right), \\
    \rho(0,x,a) = \rho_{0} (x, a).
\end{cases}
\end{equation}

Ici, $a \in \R^p$ doit être compris comme un paramètre influençant la solution de l'équation (typiquement un coefficient de diffusion, un drift constant ou une perturbation de la condition initiale). Notre objectif est d’étudier l’impact d’une variation de ce paramètre sur la solution correspondante. Nous nous intéressons donc à la dépendance du flot de cette EDP vis-à-vis de ces paramètres, que l'on mesure à l’aide de la distance de Wasserstein. Ce problème a été largement étudié pour des paramètres aléatoires dans \cite{ERNST}, où les auteurs parviennent à montrer des résultats de continuité. Dans le cas déterministe, certains résultats ont été obtenus pour d’autres EDP dans un cadre plus spécifique (par exemple dans \cite{CARLES} où les auteurs étudient la dépendance à l’exposant de non-linéarité de l’équation de Schrödinger) ou pour l’équation de Fokker-Planck mais par rapport à la norme $\mathrm{L}^2$ (voir par exemple \cite{DAVID}, \cite{ROY} ou encore \cite{KUIPER}). Notre contribution principale consiste donc à réaliser une telle analyse en utilisant la distance de Wasserstein. 
\medbreak
On suppose par la suite que $b$ et $A$ sont des fonctions Lipschitz telles que $\nabla_x \cdot A$ soit aussi Lipschitz, et qu'elles vérifient

\begin{gather*}
    |b(x,a)| + \|A(x,a)\|_F \leq C(1+ |x|), \\
    |b(x,a) - b(x',a')| + \|A(x,a) - A(x',a')\|_F \leq L_1 \left( |x-x'| + |a-a'|\right), \\
    | \nabla_x \cdot A(x,a) - \nabla_x \cdot A(x', a') | \leq L_2 \left( |x -x'| + |a-a'| \right),
\end{gather*}

de sorte à ce que \eqref{FK} soit bien posée. Nous supposons de plus que 

\begin{equation*}
    \lambda_{\min}(A(x,a) \geq m,
\end{equation*}

où $\| \cdot \|_F$ est la norme de Frobenius matricielle et $\lambda_{\min}(A(x,a))$ désigne la plus petite valeur propre de $A(x,a)$. Nous parvenons sous ces hypothèses à établir le 

\begin{theoreme}
Soient $p \in [2, + \infty)$ et $a, a' \in \R^p$. Il existe $C_{1,d,p} > 0$ et $C_{2,d,p}> 0$ deux constantes explicites telles que

\begin{equation}
    \W_p^p(\rho(t, \cdot ,a'), \rho(t, \cdot, a)) \leq \W_p^p(\rho_0(\cdot ,a'), \rho_0(\cdot, a)) e^{ C_{1,d,p}t} + \frac{C_{2,d,p}}{C_{1,d,p}} \left( e^{ C_{1,d,p} t} - 1 \right) |a' - a |^p
\end{equation}
\end{theoreme}

Nous prouvons ce résultat de deux manières différentes. La première preuve se base sur la formulation duale de la distance $\W_p$, en écrivant

\begin{equation*}
    \W_p^p \left(\rho(t, \cdot, a), \rho(t, \cdot, a') \right) = \sup_{\substack{h_{a,t}, h_{a',t} \in C_b(\R^d) \\ h_{a,t}(x) + h_{a',t}(y) \leq |x-y|^p} } \int_{\R^d} h_{t, a}(x) \rho_{t,a}(\dr x) + \int_{\R^d} h_{t, a'}(y) \rho_{t,a'}(\dr y),
\end{equation*}

puis en dérivant cette relation en utilisant un théorème d'enveloppe. La deuxième preuve se base sur un couplage synchrone entre deux EDS, c'est-à-dire via

\begin{gather*}
    \dr X_t = b(X_t, a) \dr t + \sigma(X_t, a) \dr B_t, \\
    \dr X'_t = b(X'_t, a') \dr t + \sigma(X'_t, a') \dr B_t,
\end{gather*}

où $\sigma$ est la racine carrée de $A$. En couplant les conditions intiales de manière optimale (au sens de $\W_p$) et en considérant $|X_t - X'_t|^p$, nous prouvons le résultat de manière équivalente. Dans un second temps, nous appliquons cette méthode au cas du processus de Langevin, afin de comparer la convergence vers une mesure stationnaire et l'écart dû à des paramètres différents. Plus précisément, nous considérons les EDS

\begin{gather*}
    \dr X_t = - \nabla_x V(t, X_t, a) \dr t + \sqrt{\frac{2}{\beta}} \dr B_t, \\
    \dr X'_t = - \nabla_x V(t, X'_t, a') \dr t + \sqrt{\frac{2}{\beta'}} \dr B_t,
\end{gather*}

où on suppose que 

\begin{gather*}
    \langle \nabla_x V(x,a) - \nabla_x V(y,a), x - y \rangle \geq m |x - y|^2, \\
    |\nabla_x V(x,a) - \nabla_x V (x, a') | \leq L_3 |a- a'|.
\end{gather*}

Sous ces hypothèses, nous montrons le

\begin{theoreme}
Soient $p [2, + \infty)$, $a, a' \in \R^p$ et $\beta, \beta'>0$. Il existe $\lambda >0$, $K_{1,d,p}, K_{2,d,p} >0$ telles que 

\begin{multline}
    \W_p^p(\rho(t, \cdot, a, \beta), \rho(t, \cdot, a', \beta')) \leq \W_p^p(\rho_0( \cdot, a, \beta), \rho_0(\cdot, a', \beta')) e^{- \lambda t} \\
    + \left( 1 - e^{- \lambda t} \right) \left( K_{1,d,p}|a - a'|^p + K_{2,d,p} \left|\sqrt{\frac{2}{\beta}} -  \sqrt{\frac{2}{\beta'}} \right|^p \right).
\end{multline}
\end{theoreme}

\selectlanguage{english}

\section{Background and motivations}
When studying the flow of an ODE it is natural to estimate the
variations of the solution with respect to its initial value or its
parameters. While the former has been extensively studied for many PDE such as Fokker-Planck, McKean-Vlasov, porous medium equation or
Keller-Segel equations in Wasserstein metric spaces, we are not aware of systematic studies concerning the sensibility of PDE with respect to parameters in Wasserstein metric spaces. Here, we focus on the linear Fokker-Planck equation to address its
parameter sensibility in $\W_p$ for $p$ in $[2,\infty)$. We thus consider the linear Fokker-Planck equation on $\R^d$, depending on a parameter $a \in \R^p$, with a $b : \R^d \times \R^p \rightarrow \R^d$ and a diffusion coefficient $A : \R^d \times \R^p \rightarrow \mathbb{S}_{++}^{d}(\R)$, where $\mathbb{S}_{++}^{d}(\R)$ is the set of symmetric positive definite matrix of dimension $d$.

\begin{equation}
\begin{cases}
    \partial_t \rho(t,x,a) =  \nabla_x \cdot \left(A(x,a) \nabla_x \rho(t,x,a) \right) - \nabla_x  \cdot \left(\rho(t,x,a) b(x,a)  \right), \\
    \rho(0,x,a) = \rho_{0} (x, a).
\end{cases}
    \label{FK}
\end{equation}

Here $a \in\R^p$ has to be understood as some parameter that influences the process (typically a diffusion coefficient, a constant drift or a perturbation of the initial condition). Our goal is to investigate how a change in this parameter influences the corresponding solution to \eqref{FK}. We are thus interested in the dependence of the flow of this PDE upon intrinsic parameters, with respect to the Wasserstein distance. This problem has been extensively studied for random parameters in \cite{ERNST}, showing continuity results, where the authors prove Hölder continuity results for random diffusions. In the deterministic case, some results have been proven for other PDEs in a more specific setting (see e.g. \cite{CARLES} where the authors study the dependency upon the non-linearity exponent of the Schrödinger equation) or for the Fokker-Planck equation but with respect to $\mathrm{L}^2$ norm (for instance in \cite{ARNOLD}, \cite{ROY} or \cite{KUIPER}). Our main contribution is to perform such an analysis with the Wasserstein distance.
\medbreak
Throughout this paper, we assume that $b$ and $A$ are Lipschitz functions such that $\nabla_x \cdot A$ is also Lipschitz, and that they satisfy

\begin{gather*}
    |b(x,a)| + \|A(x,a)\|_F \leq C(1+ |x|), \\
    |b(x,a) - b(x',a')| + \|A(x,a) - A(x',a')\|_F \leq L_1 \left( |x-x'| + |a-a'|\right), \\
    | \nabla_x \cdot A(x,a) - \nabla_x \cdot A(x', a') | \leq L_2 \left( |x -x'| + |a-a'| \right),
\end{gather*}

for some constants $C, L_1, L_2>0$ and where $\| \cdot \|_F$ is the Frobenius norm on matrices, so that \eqref{FK} is well-posed. We additionally assume that 

\begin{equation*}
    \lambda_{\min}(A(x,a) \geq m,
\end{equation*}

where $m$ are some positive constants and $\lambda_{\min}(A(x,a))$ stands for the smallest eigenvalue of $A(x,a)$. In order to compare two solutions, we will rely on the Wasserstein distance of order $p$, defined, for $\mu,\mu' \in \mathcal{P}(\R^d)$ with finite $p$-moments, as

\begin{equation}\label{WP}
    \begin{split}
    \W_p(\mu, \mu') &= \left( \inf_{\pi \in \Pi(\mu, \mu')} \int_{\R^d \times \R^d} \left| x- x' \right|^p \pi(\dr x,\dr x') \right)^\frac{1}{p}\\
    & = \left( \sup_{\substack{h, h' \in C_b(\R^d) \\ h(x) + h'(x') \leq |x-x'|^p} } \int_{\R^d} h(x) \mu(\dr x) + \int_{\R^d} h'(x') \mu'(\dr x') \right)^\frac{1}{p},
    \end{split}
\end{equation}

where $\Pi(\mu, \mu')$ is the set of coupling between $\mu$ and $\mu'$ and $C_b(\R^d)$ is the set of bounded continuous functions from $\R^d$ to $\R$. We are able to give an upper bound on the $\W_p$ distance between two solutions of \eqref{FK} with different parameters $a$ and $a'$. Our main result is the following theorem, for which we give two different proofs in the next sections.

\begin{theorem} \label{MAINTH}
Let $p \in [2,+\infty) $, $a, a' \in \R^p$ and $\beta, \beta'>0$. There exist $C_{1,d,p}, C_{2,d,p} >0$ two explicit constants such that 

\begin{equation}
    \W_p^p(\rho(t, \cdot ,a'), \rho(t, \cdot, a)) \leq \W_p^p(\rho_0(\cdot ,a'), \rho_0(\cdot, a)) e^{C_{1,d,p}t} + \frac{C_{2,d,p}}{C_{1,d,p}} \left( e^{C_{1,d,p} t} - 1 \right) |a' - a |^p.
\end{equation}

\end{theorem}

\begin{remark}
    In the case where $ a = a'$ (i.e. when the only uncertainty lies on the intial condition), we recover estimates that have been widely studied in the litterature (in \cite{VILLANI} for example).
\end{remark}

\begin{remark}
One of the main inconvenient of this result is that it fails to encompass the case $p \in [1,2)$ and $p= \infty$. For the case $p=1$, one may try to adapt the ideas of \cite{EBERLE}, where the author studies exponential decay w.r.t. to the $\W_1$ norm by using reflection coupling (instead of the synchronous coupling we will use thereafter). For $p =\infty$, the fact that our bound blows up when taking the $p$-th root and letting $p$ go to infinity is more concerning, which leads us to believe that the result is false when $p =\infty$.
\end{remark}

\begin{ex}
For $p=2$, most of the expressions are much simpler. Indeed, if we assume that $A$ does not depend on $x$ and that $b=0$, in this case, one simply obtains that

\begin{equation*} 
    \W_2^2(\rho(t, \cdot ,a'), \rho(t, \cdot, a)) \leq \W_2^2(\rho_0(\cdot ,a'), \rho_0(\cdot, a)) +  \frac{L_1^2}{2m} |a-a'|^2 t.
\end{equation*}

Thus, in this simplified setting, two solutions only differ by a affine in time function.
\end{ex}

\section{A first proof of Theorem 1.1. using Kantorovitch's dual formulation}

We rely on the following differential inequality, which was proven in \cite[Theorem 4.1.]{AMBROSIO} for instance, and used for many applications thereafter (see e.g. \cite{DAVID}). Recall that there exist $h_{a,t}$ and $h_{a', t}$ such that the supremum in \eqref{WP} (for the distance between $\rho(t, \cdot a)$ and $\rho(t, \cdot, a')$) is attained, and $\pi_{a, a'}$ an optimal coupling such that the infinimum is attained too. In the follwing, we write $\pi$ instead of $\pi_{a,a'}$ for the sake of clarity. It follows that

\begin{equation*}
\begin{split}
    \dfrac{\mathrm{d} }{\mathrm{d}t}\mathcal{W}_p^p(\rho(t, \cdot ,a'), \rho(t, \cdot, a))  \leq & \int_{\mathbb{R}^d} h_{a',t}(x') \partial_t \rho_{a',t} (\dr x') + \int_{\mathbb{R}^d} h_{a,t}(x) \partial_t \rho_{a,t}(\dr x) \\
 = &  \int_{\mathbb{R}^d \times \mathbb{R}^d} (\nabla_x \cdot A(x',a')) \cdot \nabla_x h_{a',t}(x') + (\nabla_x \cdot A(x,a)) \cdot \nabla_x h_{a,t}(x)   \pi(\,\mathrm{d}x', \,\mathrm{d}x) \\
& + \int_{\mathbb{R}^d \times \mathbb{R}^d} A(x',a'):\nabla_x^2 h_{a',t}(x) + A(x,a):\nabla_x^2 h_{a,t}(x) \pi(\,\mathrm{d}x', \,\mathrm{d}x) \\
&+ \int_{\mathbb{R}^d \times \mathbb{R}^d} b (x',a') \cdot \nabla_x h_{a'}(x') + b(x,a) \cdot \nabla_x h_{a}(x) \pi(\,\mathrm{d}x', \,\mathrm{d}x),
\end{split}
\end{equation*}

the last equality coming from an integration by part. On the support of $\pi$, the constraint $h_{a,t}(x) + h_{a',t}(x') \leq |x-y|^p$ is actually an equality and optimality conditions yield, on the support of $\pi$,

\begin{equation*}
    \begin{cases}
    h_{a',t}(x') + h_{a,t}(x) = | x' - x|^{p} \\
    \nabla_x h_{a',t}(x') = p|x'- x|^{p-2}(x' - x) \\
    \nabla_x h_{a,t}(x) = p |x'-x|^{p-2}(x - x') \\
    \begin{bmatrix}
    \nabla_x^2 h_{a,t}(x)& 0 \\
    0 & \nabla_x^2 h_{a',t}(x') 
    \end{bmatrix}
    \leq 
    \begin{bmatrix}
    M(x',x) & -M(x',x) \\
    - M(x',x) & M(x',x)
    \end{bmatrix},
    \end{cases}
\end{equation*}

where $M$ is given by

\begin{equation*}
    M(x',x) = p(p-2) |x' - x|^{p-4} (x'-x)(x'-x)^T + p|x' - x|^{p-2} I_d.
\end{equation*}

Since the class of symetric positive semi-definite matrices is stable by multiplication, we can multiply our last condition by the matrix $
\begin{bmatrix}
A(x,a) & \sqrt{A(x,a)} \sqrt{A(x',a')} \\
\sqrt{A(x',a')} \sqrt{A(x,a)} & A(x',a')
\end{bmatrix}
$  and take its trace, so that, on the support of $\pi$,

\begin{equation*}
\begin{split}
        A(x',a'):\nabla_x^2 h_{a'}(x) + A(x,a):\nabla_x^2 h_{a}(x) \leq &
    \text{Tr}(A(x',a') M(x',x) ) + \text{Tr}( A(x,a) M(x',x)) \\
    &- 2\text{Tr}\left(\sqrt{A(x,a)} \sqrt{A(x',a')} M(x',x)\right) \\
    & = \text{Tr}\left( \left( \sqrt{A(x',a')} - \sqrt{A(x,a)}\right)^2 M(x',x)\right) \\
    & \leq \| M(x,x') \|_{\text{op}} \text{Tr}\left( \left( \sqrt{A(x',a')} - \sqrt{A(x,a)}\right)^2 \right),
\end{split}
\end{equation*}

since $\text{Tr}(XY) \leq \|X\|_{\text{op}} \text{Tr}(Y)$, by using von Neumann's trace inequality (see for instance \cite{VONNEUMANN}), which states that 

\begin{equation*}
|\text{Tr}(XY)| \leq \sum_{i = 1}^d \sigma_i(X) \sigma_i(Y),
\end{equation*}

for every $d \times d$ complex matrices $X$ and $Y$ and where $\sigma_i(X)$ are the singular values of $X$. Since $p \geq 2$, $\| M(x,x') \|_{\text{op}} = p(p-1) |x - x'|^{p-2}$, so that we only have to handle 
\begin{equation*}
\text{Tr}\left( \left( \sqrt{A(x',a')} - \sqrt{A(x,a)}\right)^2 \right) = \left\| \sqrt{A(x',a')} - \sqrt{A(x,a)}\right\|_F^2. 
\end{equation*}

Let $B(x,a) = \sqrt{A(x,a)}$ so that $B^2(x,a) = A(x,a)$. Differentiating this equality with respect to $a$, we get that

\begin{equation*}
 \partial_a B(x,a) B(x,a) + B(x,a) \partial_a B(x,a) = \partial_a A(x,a).
\end{equation*}

It is then shown in \cite[Section 10]{BHATIA} that
\begin{equation*}
    \| \partial_{a} B(x,a) \|_F \leq \frac{\|\partial_a A(x,a)\|_F}{2 \lambda_{\min} (B(x,a) )} \leq \frac{L_1}{2 \sqrt{m}}.
\end{equation*}

It yields 

\begin{equation*}
    \| B(x,a') - B(x,a) \|_F^2 \leq |a' - a|^2 \int_0^1 \| \partial_a B(x, (1-t)a + ta') \|_F^2 \, \mathrm{d}t \leq \frac{L_1^2}{4 m} |a - a'|^2.
\end{equation*}

Applying the same method with respect to $x$ gives a similar bound, which, in the end, yields

\begin{equation*}
    \text{Tr}\left( \left( \sqrt{A(x',a')} - \sqrt{A(x,a)}\right)^2 \right) \leq \frac{L^2_1}{4m} \left( |x-x'|^2 + |a-a'|^2 \right).
\end{equation*}

Then, on the support of $\pi$,

\begin{equation*}
    A(x',a'):\nabla_x^2 h_{a'}(x) + A(x,a):\nabla_x^2 h_{a}(x) \leq \frac{L^2_1}{4m} p(p-1)  |x- x'|^p + \frac{L^2_1}{4m} p(p-1) |x - x'|^{p-2} |a - a'|^2.
\end{equation*}

We use Young's inequality on the second term, with coefficient $q = \frac{p}{2}$ and $q' = \frac{p}{p-2}$, so that 

\begin{equation} \label{FIRST}
     A(x',a'):\nabla_x^2 h_{a'}(x) + A(x,a):\nabla_x^2 h_{a}(x) \leq  \frac{L^2_1(p-1)}{2m} |a -a'|^p + \frac{L^2_1 (p-1)^2}{2m} |x -x'|^p.
\end{equation}

Our next step is to handle the drift term, by observing that, on the support of $\pi$, 

\begin{equation} \label{SECOND}
\begin{split}
     b (x',a') \cdot \nabla_x h_{a',t}(x') + b(x,a) \cdot \nabla_x h_{a,t}(x)   & =  p \langle b (x',a') - b(x,a), x' -x \rangle |x' -x|^{p-2}  \\
     &\leq p L_1 |a-a'| |x-x'|^{p-1} + p L_1 |x-x'|^p \\
     & \leq L_1 |a-a'|^p + L_1(2p-1) |x - x'|^p,
\end{split}
\end{equation}

where we used Cauchy-Schwarz inequality for the first inequality and Young's inequality for the second one. With the same method, we are also able to prove that

\begin{equation} \label{THIRD}
    (\nabla_x \cdot A(x',a')) \cdot \nabla_x h_{a',t}(x') + (\nabla_x \cdot A(x,a)) \cdot \nabla_x h_{a,t}(x) \leq L_2 |a -a'|^p + L_2(2p-1) |x-x'|^p.
\end{equation}

In the end, putting \eqref{FIRST}, \eqref{SECOND} and \eqref{THIRD} together, we get

\begin{equation}
    \dfrac{\mathrm{d} }{\mathrm{d}t}\mathcal{W}_p^p(\rho(t, \cdot ,a'), \rho(t, \cdot, a)) \leq C_{1,d,p} \W_p^p(\rho(t, \cdot ,a'), \rho(t, \cdot, a)) + C_{2,d,p} |a - a'|^p,
\end{equation}

with $C_{1,d,p} = (L_1 +L_2)(2p-1)+ \frac{L^2_1 (p-1)^2}{2m}$ and $C_{2,d,p} = L_1 + L_2 + \frac{L^2_1(p-1)}{2m}$. Using a straightforward variant of Grönwall's lemma gives the inequality of the theorem.

\section{A second proof of Theorem 1.1. using a synchronous coupling}

Let us consider $X_t$ and $X'_t$ solution to the SDEs
\begin{gather*}
    \dr X_t = \left( b(X_t,a) + \nabla \cdot A(X_t, a) \right) \dr t + \sqrt{2} \sigma(X_t, a) \dr B_t, \\
     \dr X'_t = \left( b(X'_t,a') \nabla \cdot A(X_t', a') \right) \dr t + \sqrt{2} \sigma(t, X'_t, a') \dr B_t,
\end{gather*}

where $\sigma(x,a) = \sqrt{A(x,a)}$, together with $(X_0, X_0')$ being an optimal coupling for the $p$-Wasserstein distance between $\rho_0(\cdot,a)$ and $\rho_0(\cdot,a')$. The subsequent coupling between $X_t$ and $X'_t$ is such that the underlying Brownian motion is the same in both dynamics. We now apply Ito's lemma to $|X_t - X'_t|^p$ for $p\geq2$, which yields

\begin{equation*}
\begin{split}
    \dr |X_t - X'_t|^p =& p (X_t - X'_t) \cdot (b(X_t,a) - b(X'_t, a')) |X_t - X'_t|^{p-2} \dr t \\
    & + p (X_t - X'_t) \cdot (\nabla \cdot A(X_t,a) - \nabla \cdot A(X'_t, a')) |X_t - X'_t|^{p-2} \\
     &+ p(p-2) \left|\left( \sigma(X_t,a) - \sigma(X'_t,a') \right)^T (X_t - X'_t) \right|^2 |X_t - X'_t|^{p-4} \dr t \\
     &+ p \| \sigma(X_t, a) - \sigma(X'_t, a')\|_F^2 |X_t - X'_t|^{p-2} \dr t \\
     & + p |X_t - X'_t|^{p-2} (X_t - X'_t)\cdot \left( \sigma(X_t,a) - \sigma(X'_t,a') \right) \dr B_t.
\end{split}
\end{equation*}

Taking the expectation gives

\begin{equation*}
\begin{split}
     \mathbb{E}\left[ |X_t - X'_t|^p \right] =& \mathbb{E}[|X_0 - Y_0|^p] + p \int_0^t \mathbb{E}\left[(X_s - X'_s) \cdot (b( X_s,a) - b(X'_s, a')) |X_s - X'_s|^{p-2} \right] \dr s \\
     &+  p \int_0^t \mathbb{E}\left[(X_s - X'_s) \cdot (\nabla \cdot A( X_s,a) - \nabla \cdot A(X'_s, a')) |X_s - X'_s|^{p-2} \right] \dr s\\
     &+ p(p-2)\int_0^t \mathbb{E}\left[\left|\left( \sigma(X_s,a) - \sigma(X'_s,a') \right)^T (X_s - X'_s) \right|^2 |X_s - X'_s|^{p-4} \right] \dr s \\
    & + p \int_0^t \mathbb{E}\left[ \| \sigma( X_s, a) - \sigma( X'_s, a')\|_F^2 |X_s - X'_s|^{p-2} \right] \dr s.
\end{split}
\end{equation*}

In particular,

\begin{equation*}
\begin{split}
    \frac{\dr}{\dr t} \mathbb{E}\left[ |X_t - X'_t|^p \right] =& p\mathbb{E}\left[(X_t - X'_t) \cdot (b( X_t,a) - b(X'_t, a')) |X_t - X'_t|^{p-2} \right] \\
    & + p \mathbb{E}\left[(X_t - X'_t) \cdot (\nabla \cdot A( X_t,a) - \nabla \cdot A(X'_t, a')) |X_t - X'_t|^{p-2} \right] \\
    & + p(p-2) \mathbb{E}\left[\left|\left( \sigma(X_t,a) - \sigma(X'_t,a') \right)^T (X_t - X'_t) \right|^2 |X_t - X'_t|^{p-4} \right] \\
    & + p \mathbb{E}\left[ \| \sigma( X_t, a) - \sigma( X'_t, a')\|_F^2 |X_t - X'_t|^{p-2} \right].
\end{split}
\end{equation*}

The first term and second term can be taken care of through

\begin{equation}\label{ONE}
\begin{split}
    p \mathbb{E}\left[(X_t - X'_t) \cdot (b( X_t,a) - b(X'_t, a')) |X_t - X'_t|^{p-2} \right] & \leq p L_1 \mathbb{E}\left[ |X_t - X'_t|^p \right] + pL_1 |a-a'| \mathbb{E}\left[ |X_t - X'_t|^{p-1} \right]\\
    & \leq L_1 |a-a'|^p + L_1(2p-1) \mathbb{E}\left[ |X_t - X'_t|^p \right],
\end{split}
\end{equation}

\begin{multline}\label{TWO}
    p \mathbb{E}\left[(X_t - X'_t) \cdot (\nabla \cdot A( X_t,a) - \nabla \cdot A(X'_t, a')) |X_t - X'_t|^{p-2} \right] \\
     \leq p L_2 \mathbb{E}\left[ |X_t - X'_t|^p \right] + pL_2 |a-a'| \mathbb{E}\left[ |X_t - X'_t|^{p-1} \right]\\
     \leq L_2 |a-a'|^p + L_2(2p-1) \mathbb{E}\left[ |X_t - X'_t|^p \right].
\end{multline}

Next, remark that 

\begin{equation*}
    \mathbb{E}\left[\left|\left( \sigma(X_t,a) - \sigma(X'_t,a') \right)^T (X_t - X'_t) \right|^2 |X_t - X'_t|^{p-4} \right] \leq  \mathbb{E}\left[ \| \sigma( X_t, a) - \sigma( X'_t, a')\|_F^2 |X_t - X'_t|^{p-2} \right],
\end{equation*}

so that 

\begin{multline*}
    p(p-2) \mathbb{E}\left[\left|\left( \sigma(X_t,a) - \sigma(X'_t,a') \right)^T (X_t - X'_t) \right|^2 |X_t - X'_t|^{p-4} \right] \\
    + p \mathbb{E}\left[ \| \sigma( X_t, a) - \sigma( X'_t, a')\|_F^2 |X_t - X'_t|^{p-2} \right] \\
    \leq p(p-1) \mathbb{E}\left[ \| \sigma( X_t, a) - \sigma( X'_t, a')\|_F^2 |X_t - X'_t|^{p-2} \right].
\end{multline*}

Using the previous bound on $\left\| \sqrt{A(x',a')} - \sqrt{A(x,a)}\right\|_F^2$, we also get

\begin{equation*}
    \mathbb{E}\left[ \| \sigma( X_t, a) - \sigma( X'_t, a')\|_F^2 |X_t - X'_t|^{p-2} \right] \leq \frac{L^2_1}{4m} \left( \mathbb{E}\left[ |X_t - X'_t|^p \right] + |a-a'|^2 \mathbb{E}\left[ |X_t - X'_t|^{p-2} \right] \right),
\end{equation*}

so that

\begin{equation}\label{THREE}
    p(p-1) \mathbb{E}\left[ \| \sigma( X_t, a) - \sigma( X'_t, a')\|_F^2 |X_t - X'_t|^{p-2} \right] \leq \frac{L^2_1(p-1)}{2m} |a -a'|^p + \frac{L^2_1 (p-1)^2}{2m} \mathbb{E}\left[ |X_t - X'_t|^p \right].
\end{equation}

Using \eqref{ONE}, \eqref{TWO} and \eqref{THREE} gives 

\begin{equation*}
    \frac{\dr}{\dr t} \mathbb{E}\left[ |X_t - X'_t|^p \right] \leq C_{1,d,p} \mathbb{E}\left[ |X_t - X'_t|^p \right] \dr s + C_{2,d,p} |a - a'|^p .
\end{equation*}

with $C_{1,d,p}$ and $C_{2,d,p}$ being the same constants as previously found. This is enough to conclude by Grônwall's inequality, since $\mathbb{E}[|X_0 - Y_0|^p] = \W_p^p(\rho_0(\cdot ,a'), \rho_0(\cdot, a))$ by construction.

\section{Further results for the overdamped Langevin process}

In this section, we consider a slight variation of our original PDE, in order to highlight what can happen when the underlying PDE converges towards a steady-state. We chose the overdamped Langevin process for the sake of clarity, but the method is reliable enough so that it can be adapted to various similar settings, such as developed in \cite{BOLLEY}. Let us consider 

\begin{equation}
\begin{cases}
    \partial_t \rho(t,x,a, \beta) =  \frac{1}{\beta} \Delta \rho(t,x,a, \beta) + \nabla_x  \cdot \left(\rho(t,x,a, \beta) \nabla_x V(x,a)  \right), \\
    \rho(0,x,a, \beta) = \rho_{0} (x, a, \beta).
\end{cases}
    \label{LANGEVIN}
\end{equation}

Observe that we chose our parameter to be the thermal energy $\beta>0$ and some parameter $a \in \R^p$ only influencing the potential $V$. We assume in the following that 

\begin{gather*}
    \langle \nabla_x V(x,a) - \nabla_x V(y,a), x - y \rangle \geq k |x - y|^2, \\
    |\nabla_x V(x,a) - \nabla_x V (x, a') | \leq L_3 |a- a'|,
\end{gather*}

for some $k, L_3 >0$. Under this set of hypothesis, we prove the following 

\begin{theorem}

Let $p \in [2,+ \infty)$, $a, a' \in \R^p$ and $\beta, \beta'>0$. There exist $\lambda, K_{1,d,p}, K_{2,d,p} >0$ such that 
\begin{multline}
    \W_p^p(\rho(t, \cdot, a, \beta), \rho(t, \cdot, a', \beta')) \leq \W_p^p(\rho_0( \cdot, a, \beta), \rho_0(\cdot, a', \beta')) e^{- \lambda t} \\
    + \frac{1}{\lambda} \left( K_{1,d,p}|a - a'|^p + K_{2,d,p} \left|\sqrt{\frac{2}{\beta}} -  \sqrt{\frac{2}{\beta'}} \right|^p \right) \left( 1 - e^{- \lambda t} \right).
\end{multline}
\end{theorem}

\begin{proof}
We proceed as before, by using a coupling argument. We introduce

\begin{gather*}
    \dr X_t = - \nabla_x V(t, X_t, a) \dr t + \sqrt{\frac{2}{\beta}} \dr B_t, \\
    \dr X'_t = - \nabla_x V(t, X'_t, a') \dr t + \sqrt{\frac{2}{\beta'}} \dr B_t,
\end{gather*}

Applying Ito's lemma to $|X_t - X'_t|^p$ yields

\begin{equation*}
\begin{split}
    \dr |X_t - X'_t|^p =& - p (X_t - X'_t) \cdot (\nabla_x V(X_t,a) - \nabla_x V(X'_t, a')) |X_t - X'_t|^{p-2} \dr t \\
     &+ \frac{p(p-2+d)}{2} \left|\sqrt{\frac{2}{\beta}} -  \sqrt{\frac{2}{\beta'}} \right|^2 |X_t - X'_t|^{p-2} \dr t \\
     & + p |X_t - X'_t|^{p-2} \left(\sqrt{\frac{2}{\beta}} -  \sqrt{\frac{2}{\beta'}} \right) (X_t - X'_t)\cdot \dr B_t.
\end{split}
\end{equation*}

But, by convexity,

\begin{equation*}
    - p \mathbb{E}\left[ (X_t - X'_t) \cdot (\nabla_x V(X_t,a) - \nabla_x V(X'_t, a)) |X_t - X'_t|^{p-2} \right] \leq - p k \mathbb{E}\left[  |X_t - X'_t|^p\right],
\end{equation*}

and, since $\nabla_x V$ is Lipschitz with respect to $a$,

\begin{multline*}
     - p \mathbb{E}\left[ (X_t - X'_t) \cdot (\nabla_x V(X'_t,a) - \nabla_x V(X'_t, a')) |X_t - X'_t|^{p-2} \right] \\
      \leq p L_3  |a-a'|\mathbb{E}\left[ |X_t - X'_t|^{p-1} \right] \\
     \leq \frac{pk}{4}\mathbb{E}\left[ |X_t - X'_t|^{p} \right] + \frac{(4 (p-1))^{p-1} L_3^p}{(p k)^{p-1}} |a-a'|^p,
\end{multline*}

using that $ab \leq \frac{\varepsilon^q}{q} a^q + \frac{1}{\varepsilon^{q'} q'} b^{q'}$ for $\varepsilon = \left( \frac{p m}{4 (p-1) L_3} \right)^{\frac{p-1}{p}}$. With the same kind of argument, we get that

\begin{multline*}
    \frac{d p(p-1)}{2} \left|\sqrt{\frac{2}{\beta}} -  \sqrt{\frac{2}{\beta'}} \right|^2 \mathbb{E} \left[ |X_t - X'_t|^{p-2} \right] \\
     \leq \frac{pk}{4}\mathbb{E}\left[ |X_t - X'_t|^{p} \right] + (p-1) d \left( \frac{2 d (p-1)(p-2)}{k} \right)^{\frac{p}{2}} \left|\sqrt{\frac{2}{\beta}} -  \sqrt{\frac{2}{\beta'}} \right|^p.
\end{multline*}

Putting everything together, we have

\begin{equation*}
    \frac{\dr}{\dr t} \mathbb{E}\left[ |X_t - X'_t|^p \right] \leq - \lambda \mathbb{E}\left[ |X_t - X'_t|^p \right] + K_{1,d,p} |a - a'|^p + K_{2,d,p} \left|\sqrt{\frac{2}{\beta}} -  \sqrt{\frac{2}{\beta'}} \right|^p, 
\end{equation*}

with $\lambda = \frac{pk}{2}$, $K_{1,d,p} = \frac{(4 (p-1))^{p-1} L_3^p}{(p k)^{p-1}}$ and $K_{2,d,p} = (p-1) d \left( \frac{2 d (p-1)(p-2)}{k} \right)^{\frac{p}{2}} $, which is exactly the inequality we were looking for.

\end{proof}

\section*{Acknowledgments}

The author would like to thank Vincent Calvez for valuable feedback on the manuscript. This project has received funding from the European Research Council (ERC) under the European Union’s Horizon 2020 research and innovation programme (grant agreement No 865711).

\section*{Conflicts of interest}

The authors do not work for, advise, own shares in, or receive funds from any organization
that could benefit from this article, and have declared no affiliations other than their research
organizations.



\begin{thebibliography}{99}

\bibitem{BHATIA}
R.~Bhatia and P.~Rosenthal,
\newblock How and why to solve the operator equation $AX - XB = Y$,
\newblock \emph{Bull. London Math. Soc.}, 29(1):1--21, 1997.

\bibitem{AMBROSIO}
L.~Ambrosio, N.~Gigli, and G.~Savar{\'e},
\newblock Metric measure spaces with Riemannian Ricci curvature bounded from below,
\newblock \emph{Duke Math. J.}, 163(7):1405--1490, 2014.

\bibitem{BOLLEY}
F.~Bolley, I.~Gentil, and A.~Guillin,
\newblock Convergence to equilibrium in Wasserstein distance for Fokker--Planck equations,
\newblock \emph{J. Funct. Anal.}, 263(8):2430--2457, 2012.

\bibitem{ERNST}
O.~G. Ernst, A.~Pichler, and B.~Sprungk,
\newblock Wasserstein sensitivity of risk and uncertainty propagation,
\newblock \emph{SIAM/ASA J. Uncertain. Quantif.}, 10(3):915--948, 2022.

\bibitem{ROY}
S.~Roy, S.~Pal, A.~Manoj, S.~Kakarla, J.~V. Padilla, and M.~Alajmi,
\newblock A Fokker--Planck framework for parameter estimation and sensitivity analysis in colon cancer,
\newblock \emph{AIP Conf. Proc.}, 2522(1):070005, 2022.

\bibitem{CARLES}
R.~Carles, Q.~Chauleur, and G.~Ferriere,
\newblock On the dependence of the nonlinear Schr\"odinger flow upon the power of the nonlinearity,
\newblock Preprint, 46 pages, 2025.

\bibitem{KUIPER}
K.~Elamvazhuthi, H.~Kuiper, and S.~Berman,
\newblock Controllability to equilibria of the 1-D Fokker--Planck equation with zero-flux boundary condition,
\newblock in \emph{Proc. IEEE 56th Annual Conf. on Decision and Control (CDC)},
Melbourne, Australia, pp.~2485--2491, 2017.

\bibitem{VONNEUMANN}
L.~Mirsky,
\newblock A trace inequality of John von Neumann,
\newblock \emph{Monatsh. Math.}, 79(4):303--306, 1975.

\bibitem{EBERLE}
A.~Eberle,
\newblock Reflection couplings and contraction rates for diffusions,
\newblock \emph{Probab. Theory Related Fields}, 166(3):851--886, 2016.

\bibitem{ARNOLD}
A.~Arnold, S.~Jin, and T.~W\"ohrer,
\newblock Sharp decay estimates in local sensitivity analysis for evolution equations with uncertainties:
From ODEs to linear kinetic equations,
\newblock \emph{J. Differential Equations}, 268(3):1156--1204, 2020.

\bibitem{DAVID}
N.~David, A.~R. M\'esz\'aros, and F.~Santambrogio,
\newblock Improved convergence rates for the Hele--Shaw limit in the presence of confining potentials,
\newblock \emph{J. \'Ecole Polytech. Math.}, 13:41--71, 2025.

\bibitem{VILLANI}
C.~Villani,
\newblock \emph{Optimal Transport: Old and New},
\newblock Grundlehren der mathematischen Wissenschaften, vol.~338,
Springer, Berlin, 2009.

\end{thebibliography}
\end{document}